\documentclass[12pt]{article}
\usepackage{amsmath, amssymb, amsthm, enumerate, hhline, fullpage, bbm, eucal, cancel, stackrel}
\usepackage{tikz-cd}
\usetikzlibrary{arrows}
\pdfoutput=1
\usepackage[all,cmtip]{xy}

\usepackage[hidelinks]{hyperref}

\newtheoremstyle{mystyle_definition}{10pt}{10pt}{}{}{\normalfont\bfseries}{.}{4pt}{\textbf{\thmname{#1}\hspace{3pt}\thmnumber{#2}\hspace{3pt}\thmnote{(#3)}}}
\theoremstyle{plain}
\newtheorem{proposition}{Proposition}
\newtheorem{theorem}{Theorem}
\newtheorem{lemma}{Lemma}

\theoremstyle{mystyle_definition}
\newtheorem{definition}{Definition}
\theoremstyle{remark}
\newtheorem{remark}{Remark}

\begin{document}

\vskip 1.5 true cm

\begin{center}  

{\Large \textbf{The para-HK/QK correspondence}}\\[.5em] 

\vskip 0.5 true cm  

{\large Malte Dyckmanns$^1$ and Owen Vaughan$^{1,2}$} \\

\small

\vskip 0.3 true cm  

$^1${Department of Mathematics and Center for Mathematical Physics\\ 
Universit\"at Hamburg, Bundesstra{\ss}e 55, D-20146 Hamburg, Germany\\  
malte.dyckmanns@math.uni-hamburg.de\\
owen.vaughan@math.uni-hamburg.de}

\vskip 0.3 true cm 

$^2${Department of Mathematics, King's College London\\ 
Strand, London WC2R 2LS, UK\\  
owen.vaughan@kcl.ac.uk}\\

\vskip 0.3 true cm  

 \today

\vskip 0.5 true cm  

\normalsize

\end{center}

\begin{abstract}

\noindent  We generalise the hyper-K\"ahler/quaternionic K\"ahler (HK/QK) correspondence to include para-geometries, 
and present a new concise proof that the target manifold of the HK/QK correspondence is quaternionic K\"ahler.
As an application, we construct one-parameter deformations of the temporal and Euclidean supergravity $c$-map metrics and show that they are para-quaternionic K\"ahler.

\end{abstract}

\section*{Introduction}

In the original HK/QK correspondence developed by Haydys in \cite{Ha} one starts with a hyper-K\"ahler manifold endowed with a \emph{rotating}\footnote{This means that the vector field preserves one of the three complex structures while acting as an infinitesimal rotation on the other two.} Killing vector field and constructs a conical hyper-K\"ahler manifold of four dimensions higher, such that the original manifold can be recovered via the hyper-K\"ahler quotient construction\footnote{With non-zero choice of level set for the homogeneous hyper-K\"ahler moment map.}. 
This conical hyper-K\"ahler manifold is then locally a Swann bundle over a quaternionic K\"ahler manifold which is again of four dimensions lower, i.e.\ of the same dimension as the original hyper-K\"ahler manifold. The construction of this quaternionic K\"ahler manifold from the original hyper-K\"ahler manifold is called the HK/QK correspondence.
This correspondence was generalised in \cite{ACM,ACDM} to include quaternionic K\"ahler manifolds of negative scalar curvature as well as pseudo-Riemannian hyper-K\"ahler and quaternionic K\"ahler manifolds, and has been formulated in terms of the associated twistor spaces in \cite{Hi}.

The main goal of this paper is to generalise the HK/QK correspondence to include para-versions of hyper-K\"ahler and quaternionic K\"ahler geometry. These are defined as follows\footnote{%
An equivalent definition of a para-hyper-K\"ahler manifold of any dimension and a para-quaternionic K\"ahler manifold of dimension larger than $4$ is that the holonomy group is contained in $Sp(2n, {\mathbbm R}) \subset SO(2n,2n)$ and $Sp(2n, {\mathbbm R})  \cdot Sp(2, {\mathbbm R}) \subset SO(2n,2n)$ respectively \cite{AC1}.}.

\begin{definition}
	A \emph{para-hyper-K\"ahler}  manifold $(M,g, J_1, J_2, J_3)$ is a pseudo-Riemannian manifold $(M,g)$ endowed with three skew-symmetric endomorphism fields $J_1, J_2, J_3 \in \Gamma(\text{End} \,TM)$ 
that  satisfy the para-quaternion algebra
	\begin{equation}
		J_1^2 = \epsilon_1 Id_{TM} \;, \qquad
		J_2^2 = \epsilon_2 Id_{TM} \;, \qquad
		J_1^2 = \epsilon_3 Id_{TM} \;, \qquad
		J_1 J_2 = J_3 \;, \label{eq:alg}
	\end{equation}
where $(\epsilon_1,\epsilon_2,\epsilon_3)$ is a permutation of $(-1,1,1)$, and such that the corresponding fundamental two-forms are closed. 
\end{definition}

\begin{definition}
	A \emph{para-quaternionic K\"ahler} manifold $(M,g,Q)$ of dimension $4n > 4$ is a pseudo-Riemannian manifold $(M,g)$ with non-zero scalar curvature endowed with a parallel skew-symmetric rank-three subbundle $Q \subset \text{End}\, TM$ that is locally spanned by three endomorphism fields $J_1, J_2, J_3 \in \Gamma(Q)$ that satisfy the para-quaternion algebra \eqref{eq:alg}.
	\label{eq:PQKdef}
\end{definition}

The curvature tensor of a para-quaternionic K\"ahler manifold of dimension $4n > 4$ admits the decomposition \cite{AC1}
	\begin{equation}
		R = \nu R_0 + W \;,
		\label{eq:RQK}
	\end{equation}
	where $\nu := scal/(4n(n+2))$, $R_0$ is the curvature tensor of para-quaternionic projective space $HP^n$ and $W$ is a trace-free $Q$-invariant algebraic curvature tensor. 
In dimension four we define a para-quaternionic K\"ahler manifold to be a pseudo-Riemannian manifold satisfying Definition \ref{eq:PQKdef} and that admits the decomposition \eqref{eq:RQK} of the curvature tensor.

In this paper we will construct what will be referred to as the \emph{para-hyper-K\"ahler/para-quaternionic K\"ahler (para-HK/QK) correspondence}. This construction maps a para-hyper-K\"ahler manifold of dimension $4n$ to a para-quaternionic K\"ahler manifold of the same dimension. We will see that many of the arguments involved in the HK/QK correspondence can be directly applied to the para-HK/QK correspondence by flipping certain signs. We will therefore present a unified discussion of both correspondences which we will refer to jointly as the $\varepsilon$-HK/QK correspondence, where the parameter $\varepsilon$ distinguishes between the two correspondences according to the rule
\[
	\varepsilon = 
	\begin{cases}
		-1 & \text{HK/QK correspondence} \\
		+1 & \text{Para-HK/QK correspondence} \;.
	\end{cases}
\]
Similarly, we will use the terminology $\varepsilon$-hyper-K\"ahler and $\varepsilon$-quaternionic K\"ahler to refer to a hyper-K\"ahler or quaternionic K\"ahler manifold in the case $\varepsilon = -1$, and a para-hyper-K\"ahler or para-quaternionic K\"ahler manifold in the case $\varepsilon = 1$. 
We will formulate the $\varepsilon$-HK/QK correspondence in terms of a rank-one principal bundle $P$ over the $\varepsilon$-hyper-K\"ahler base manifold $M$, rather than considering a conical $\varepsilon$-hyper-K\"ahler manifold. The $\varepsilon$-quaternionic K\"ahler target manifold $M'$ is then a codimension one submanifold of $P$. 
This is summarised in the following diagram:
\[
\begin{tikzcd}[column sep=6mm, row sep = 5mm]
		& & P_{4n + 1} \arrow{dll}  & & \\
		M_{4n} \arrow[mapsto]{rrrr}[swap]{\text{$\varepsilon$-HK/QK}} & & & & M'_{4n} \arrow[left hook->,swap]{ull}
\end{tikzcd}
\]
Using the rank-one principle bundle $P$ over $M$, we prove that $M'$ inherits an $\varepsilon$-quaternionic K\"ahler structure. This gives, in particular, a new concise proof of the original HK/QK correspondence. Compared to \cite{Ha,ACM,ACDM} this proof is closer to that of \cite{MS1,MS2} where the HK/QK correspondence is incorporated into Swann's twist formalism.

In the para-HK/QK correspondence the rotating Killing vector field on the para-hyper-K\"ahler base manifold preserves either a complex or a para-complex structure.
In this sense the $\varepsilon$-HK/QK correspondence can be split into three distinct subcases:
\begin{enumerate}[(i)]
	\item The HK/QK correspondence induced by a holomorphic vector field.
	\item The para-HK/QK correspondence induced by a holomorphic vector field.
	\item The para-HK/QK correspondence induced by a para-holomorphic vector field.
\end{enumerate}
We will show that in cases (i) and (ii) the $\varepsilon$-quaternionic K\"ahler target manifold admits an integrable complex structure, whilst in case (iii) the para-quaternionic K\"ahler target manifold admits an integrable para-complex structure. In all cases the integrable structure is induced by the structure on the base manifold that is preserved by the rotating Killing vector field, and is compatible with the $\varepsilon$-quaternionic structure.

Para-quaternionic K\"ahler manifolds have recently appeared in the physics literature in the context of the \emph{local temporal} (supergravity) \emph{$c$-map} and  the \emph{local Euclidean} (supergravity) \emph{$c$-map} \cite{Cortes:2015wca}.
The local temporal $c$-map is a map from a projective special K\"ahler manifold of dimension $2n$ to a para-quaternionic K\"ahler manifold of dimension $4n + 4$ that is induced by the dimensional reduction of 4D, ${\cal N} = 2$ Minkowskian local vector-multiplets over a timelike circle. 
Similarly, the local Euclidean $c$-map is a map from a projective special para-K\"ahler manifold to a para-quaternionic K\"ahler manifold (with the same dimensions as above) that is induced by the dimensional reduction of 4D, ${\cal N} = 2$ Euclidean local vector-multiplets over a spacelike circle. 
We will see that both maps can be understood geometrically in terms of the para-HK/QK correspondence, with the local temporal $c$-map corresponding to case (ii) and the local Euclidean $c$-map corresponding to case (iii). 
This provides an alternative proof that the target manifolds in both cases are para-quaternionic K\"ahler.
Moreover, we will use the para-HK/QK correspondence to construct one-parameter deformations of the local temporal and Euclidean $c$-map metrics, which are para-quaternionic K\"ahler by construction. This is analogous to the proof given in \cite{ACDM} that the one-loop deformed local spatial $c$-map metric, which first appeared in the physics literature in \cite{RoblesLlana:2006ez}, is quaternionic K\"ahler\footnote{The fact that the rigid $c$-map metric can be obtained from the one-loop deformed local spatial $c$-map metric via the QK/HK correspondence was previously shown in \cite{APP}.}. To our knowledge the deformations of the local temporal and Euclidean $c$-map metrics that we present here have not previously appeared in the literature.

\subsection*{Acknowledgements}

We would like to thank Vicente Cort\'es for suggesting the topic of this paper and for useful discussions.
The work of M.D.\ was supported by the RTG 1670 ``Mathematics inspired by String Theory,'' funded by the Deutsche Forschungsgemeinschaft (DFG).
The work of O.V.\ was supported by the German Science Foundation (DFG) under the Collaborative Research Center
(SFB) 676 ``Particles, Strings and the Early Universe.''

\section{The $\varepsilon$-HK/QK correspondence}\label{secPHKQK}

Let $(M,g,J_1,J_2,J_3)$ be an $\varepsilon$-hyper-K\"ahler manifold with $\varepsilon$-hyper-complex structure satisfying \eqref{eq:alg} with $(\epsilon_1,\epsilon_2,\epsilon_3)$ a permutation of $(-1,\varepsilon,\varepsilon)$. 
We will use the following convention for the definition of the $\epsilon_\alpha$-K\"ahler forms\footnote{Note that this convention differs from the convention in \cite{Cortes:2015wca} by a minus sign.}:
\begin{equation}
	\omega_\alpha := -\epsilon_\alpha g(J_\alpha \cdot, \cdot) \;, \qquad \alpha = 1,2,3 \;.
\end{equation}
Notice that 
\begin{equation}
	\epsilon_3 = -\epsilon_1 \epsilon_2 \;.
\end{equation}
This means that $J_1, J_2, J_3$ are complex or para-complex according to the rule
	\begin{center}
		\begin{tabular}{|c|c|c|c|}
			\hline
			 & $(\epsilon_1, \epsilon_2)$ & complex & para-complex \\
			\hhline{|-|-|-|-|}
			(i) & $(-1,-1)$ & $J_1, J_2, J_3$ &  \\
			(ii) & $(-1,+1)$ & $J_1$ & $J_2, J_3$  \\
			(iii) & $(+1,-1)$ & $J_2$ & $J_1, J_3$ \\
			(iv) & $(+1, +1)$ & $J_3$ & $J_1, J_2$  \\
			\hline 
		\end{tabular}
	\end{center}
and fulfil
\begin{equation} J_\alpha J_\beta=-J_\beta J_\alpha=\epsilon_3\epsilon_\gamma J_\gamma \end{equation}
for any cyclic permutation $(\alpha,\beta,\gamma)$ of $(1,2,3)$. 
The ordering of the endomorphisms in the above table will be important later. In particular, the $\epsilon_1$-complex structure $J_1$ induces an integrable $\epsilon_1$-complex structure on the $\varepsilon$-quaternionic K\"ahler target manifold. This means that for the para-HK/QK correspondence there are two possibilities: either the para-quaternionic K\"ahler target manifold admits an integrable complex structure, which corresponds to case (ii), or an integrable para-complex structure, which corresponds to cases (iii) and (iv). In the discussion that follows there will be no distinction between cases (iii) and (iv), which are equivalent up to relabelling of $J_2$ and $J_3$. The three distinct cases (i), (ii) and (iii) correspond to the three subcases of the $\varepsilon$-HK/QK correspondence described in the Introduction.

In order to perform the $\varepsilon$-HK/QK correspondence there must exist a real-valued function $f \in { C}^\infty(M)$ such that the vector field
\begin{equation}
\qquad\qquad\qquad\qquad\;\;\,\qquad\qquad Z := -\omega_1^{-1}(df)\qquad\qquad (\text{i.e.\ }\omega_1(Z,\cdot)=-df)
\end{equation}
is timelike or spacelike, Killing, $J_1$-holomorphic (that is, ${\cal L}_Z J_1 = 0$), and satisfies 
\begin{equation}
	{\cal L}_Z J_2 = \epsilon_1 2J_3 \;.
\end{equation}
Such a vector field is called a \emph{rotating} vector field.
We define 
\begin{equation}
	f_1 := f - \frac{g(Z,Z)}{2} \;,
	\qquad
	\beta := g(Z,\cdot) \;,
	\qquad
	\sigma := \text{sign} f \;, \qquad \sigma_1 := \text{sign} f_1 \;,
\end{equation}
and assume that $\sigma$ and $\sigma_1$ are constant and non-vanishing. 

Up to a minus sign, the function $f$ is an $\epsilon_1$-K\"ahler moment map of $Z$ with respect to $\omega_1$.
Note the simple but important fact that such a moment map is only defined up to a shift by a real constant. This will lead to a one-parameter deformation of the resulting $\varepsilon$-quaternionic K\"ahler metric.

Let $\pi:P \to M$ be a rank-one principal bundle over $M$, and let $\eta\in\Omega^1(P)$ be a principal connection on $P$ with curvature
\begin{equation}
	d\eta = \pi^*\left(\omega_1 - \frac12 d\beta\right)\;.
\end{equation}
Let $\tilde{Y}\in \Gamma(\mathrm{ker}\,\eta)$  denote the horizontal lift to $P$ of any vector field $Y$ on $M$, and let $X_P$ denote the fundamental vector field of the principal action on $P$ normalised such that $\eta(X_P) = 1$.
On $P$ we define the metric
\begin{equation}
	g_P := \frac{2}{f_1} \eta^2 + \pi^* g \;,
	\label{eq:gp}
\end{equation}
the vector field
\begin{equation}
	Z_1^P := \tilde{Z} + f_1 X_P \;,
\end{equation}
and one-forms 
\begin{align}\nonumber
	\theta_0^P &:= \frac12 df \;, \\\nonumber
	\theta_1^P &:= \eta + \frac12 \beta   \;, \\\nonumber
	\theta_2^P &:= - \frac{\epsilon_2}2 \pi^* \omega_3(Z,\cdot) \;, \\
	\theta_3^P &:=  \frac{\epsilon_2}2 \pi^*  \omega_2(Z, \cdot)\;.\label{eqDefThetaP}
\end{align}
Here we do not explicitly write the pull-back symbol in front of the functions $f, f_1, \beta$. 

In the following, we will prove in particular that
\begin{align}	
		g' &:= \left. \frac{1}{2|f|} \left( g_P - \frac{2}{f} \left( (\theta_1^P)^2 - \epsilon_1 (\theta_0^P)^2 - \epsilon_2 (\theta_3^P)^2 - \epsilon_3 (\theta_2^P)^2 \right) \right) \right|_{M'}\nonumber\\
		&\,= \frac{1}{2|f|}\left. \left(g_P-\frac{2\epsilon_1}{f}\sum_{a=0}^3 \epsilon_a(\theta_a^P)^2\right)\right|_{M'}\qquad\qquad\qquad (\epsilon_0:=-1)\label{eqDefQKMetricHKQK}
	\end{align}
defines an $\varepsilon$-quaternionic K\"ahler metric on any codimension one submanifold $M'\subset P$ that is transversal to $Z_1^P$.
	
For the proof of the above statement, we will now gather some relevant geometric properties of the original $\varepsilon$-hyper-K\"ahler manifold $M$ that are implied by the existence of the rotating Killing vector field $Z$.

Using $Z$ one may define a rank-four `vertical' distribution and its `horizontal' orthogonal complement in $TM$:
\[
	{\cal D}^v := \text{span}\{ Z, J_1 Z, J_2 Z, J_3 Z \} \;,
	\qquad
	{\cal D}^h := ({\cal D}^v)^{\perp_g}\subset TM \;,
\]
in which case the tangent bundle decomposes as 
\[
	TM = {\cal D}^v \oplus^{\perp_g} {\cal D}^h \;.
\]
With respect to the frame $(Z,J_1Z,J_2Z,J_3Z)$ on ${\cal D}^v$, the endomorphisms $J_1, J_2, J_3$ are represented respectively by the matrices
\begin{equation}
	\left( \begin{array}{cccc}
		0 & \epsilon_1 & 0 & 0 \\
		1 & 0 & 0 & 0 \\
		0 & 0 & 0 & \epsilon_1 \\
		0 & 0 & 1 & 0 
	\end{array} \right) \;, 
	\quad
	\left( \begin{array}{cccc}
		0 & 0 & \epsilon_2 & 0 \\
		0 & 0 & 0 & -\epsilon_2 \\
		1 & 0 & 0 & 0 \\
		0 & -1 & 0 & 0 
	\end{array} \right) \;, 
	\quad
	\left( \begin{array}{cccc}
		0 & 0 & 0 & -\epsilon_1 \epsilon_2 \\
		0 & 0 & \epsilon_2 & 0 \\
		0 & -\epsilon_1  & 0 & 0 \\
		1 & 0 & 0 & 0 
	\end{array} \right) \;.
	\label{eq:matrices}
\end{equation}
Let us define the following one-forms on $M$:
\begin{align*}
	\theta_0 &:= \frac12 df = -\frac12 \omega_1(Z,\cdot) = \frac{\epsilon_1}2 g(J_1Z, \cdot)  
	& &= \frac{\epsilon_1}{2} g(W, \cdot) \;, \\
	\theta_1 &:= \frac12 \beta = \frac12 g(Z,\cdot)  
	& &= \frac{\epsilon_1}{2} g(J_1 W, \cdot)  \;, \\
	\theta_2 &:= -\frac{\epsilon_2}2 \omega_3(Z,\cdot) = -\frac{\epsilon_1}2 g(J_3Z, \cdot)  
	& &= \frac{\epsilon_1}{2} g(J_2 W, \cdot)  \;, \\
	\theta_3 &:=\frac{ \epsilon_2}2\omega_2(Z, \cdot) = -\frac12 g(J_2Z, \cdot)  
	& &= \frac{\epsilon_1}{2} g(J_3 W, \cdot)  \;,
\end{align*}
where $W := J_1 Z$. Notice that the one-forms $(\theta_a^P)_{a=0,\ldots,3}$ defined in Eq.\ \eqref{eqDefThetaP} are precisely the pull-backs of the one-forms $(\theta_a)$ defined above with the exception of $\theta_1^P$ which has an additional $\eta$ inserted.

\begin{proposition}
	The $\varepsilon$-hyper-K\"ahler metric can be written as
	\begin{equation}
		g = \frac{4}{\beta(Z)} \left( (\theta_1)^2 - \epsilon_1 (\theta_0)^2 - \epsilon_2 (\theta_3)^2 - \epsilon_3 (\theta_2)^2\right)  + \check{g} \;,
		\label{eq:gdecomp}
	\end{equation}
	where $\check{g}$ is a symmetric rank-two tensor field that is invariant under $Z$ and has four-dimensional kernel ${\cal D}^v$.
	The $\varepsilon$-K\"ahler forms are given by
	\begin{equation}
			\omega_\alpha := -\epsilon_\alpha  g(J_\alpha  \cdot, \cdot) 
			= \frac{4}{\beta(Z)} \left( \epsilon_1 \epsilon_\alpha \theta_0 \wedge \theta_\alpha - \epsilon_2 \theta_\beta \wedge \theta_\gamma \right) + \check{\omega}_\alpha \;,
	\label{eq:omegaalpha}
	\end{equation}
		where $(\alpha,\beta,\gamma)$ is a cyclic permutation of $(1,2,3)$, and we have defined
	$\check{\omega}_\alpha := -\epsilon_\alpha \check{g}(J_\alpha \cdot, \cdot)$.
	The two-forms $(\check{\omega}_\alpha)_{\alpha = 1,2,3}$ are degenerate but not invariant under $Z$.
\end{proposition}

\begin{proof}
	Since $g$ is $\varepsilon$-hyper-Hermitian $Z, J_1Z, J_2Z,J_3Z$ are pairwise orthogonal. They have squared norm $g(Z,Z) = \beta(Z), -\epsilon_1 \beta(Z), - \epsilon_2 \beta(Z),- \epsilon_3 \beta(Z)$ respectively.
The tensor field
	\begin{align*}
		\check{g} &= g - \frac{4}{\beta(Z)} \left( (\theta_1)^2  -\epsilon_1 (\theta_0)^2 - \epsilon_2 (\theta_3)^2 - \epsilon_3 (\theta_2)^2\right) \\
		&= g - \frac{1}{g(Z,Z)} \left( (Z^\flat)^2  - \epsilon_1 (J_1Z^\flat)^2 - \epsilon_2 (J_2Z^\flat)^2 - \epsilon_3 (J_3 Z^\flat)^2\right) \;,
	\end{align*}
where $Z^\flat = g(Z, \cdot)$ and $J_\alpha Z^\flat = g(J_\alpha Z, \cdot)$,	has kernel ${\cal D}^v$. Since $Z$ is Killing we have  ${\cal L}_Z \beta = 0 = {\cal L}_Z\theta_1$, and since $Z$ is $J_1$-holomorphic we have ${\cal L}_Z \theta_0 = 0$. In addition we have ${\cal L}_Z J_2 = \epsilon_1 2 J_3$ and ${\cal L}_Z J_3 = 2 J_2$ and therefore 
	\[
		{\cal L}_Z \theta_2 = \epsilon_1 2 \theta_3 \;,
		\qquad
		{\cal L}_Z \theta_3 = 2\theta_2 \;.
	\]
	From these expressions it easily follows that ${\cal L}_Z \check{g} = 0 \;.$
	
	Next, we calculate 
	\[
		J^*_\alpha \theta_0 = -\theta_\alpha\;,
		\qquad
		J_\alpha^\ast\theta_\beta=-\epsilon_3\epsilon_\gamma \theta_\gamma\;,
		\qquad
		J_\alpha^\ast\theta_\alpha=-\epsilon_\alpha \theta_0 \;,
		\qquad
		J^*_\alpha \theta_\gamma = \epsilon_3\epsilon_\beta \theta_\beta \;
	\]
	for any cyclic permutation $(\alpha,\beta,\gamma)$ of $(1,2,3)$.
	Using the fact that $\epsilon_\alpha\epsilon_\beta=-\epsilon_\gamma$ the expressions for the $\varepsilon$-K\"ahler forms  are given by
	\begin{align*}
		\omega_\alpha &
		=- \frac{4\epsilon_1\epsilon_\alpha}{\beta(Z)} \left( -\theta_0 \wedge \theta_\alpha - \epsilon_3 \epsilon_\alpha \theta_\beta \wedge \theta_\gamma \right) + \check{\omega}_\alpha \;,
	\end{align*}
	which can be written as \eqref{eq:omegaalpha}.
\end{proof}

~

Let us now turn our attention to the rank-one principal bundle $\pi: P \to M$.
\begin{lemma}
	The exterior derivatives of the one-forms $\left(\theta^P_\alpha\right)_{\alpha = 1,2,3}$ defined on $P$ in Eq.\ \eqref{eqDefThetaP} are given by
	\[
		d\theta^P_\alpha = \epsilon_1 \epsilon_\alpha \pi^* \omega_\alpha \;, \qquad \alpha=  1,2,3 \;.
	\]
	\label{lem:theta}
\end{lemma}
\begin{proof}
	From the curvature of $\eta$ we see immediately that
	\[
		d\theta_1^P = d\eta + \frac12 d\beta = \pi^* \omega_1 = \epsilon_1 \epsilon_1 \pi^* \omega_1 \;.
	\]
	Next, note that
	\[
		{\cal L}_Z \omega_3 = -\epsilon_1 2 \omega_2 \;,
		\qquad
		{\cal L}_Z \omega_2 = -2  \omega_3 \;,
	\]
	from which it follows that 
	\begin{align*}
		d\theta_2^P &= - \frac{\epsilon_2}2 d(\iota_Z \omega_3) = - \frac{\epsilon_2}2 {\cal L}_Z \omega_3 = \epsilon_1 \epsilon_2 \omega_2  \;, \\
		d\theta_3^P &=  \frac{\epsilon_2}2 d(\iota_Z \omega_2) = \frac{\epsilon_2}2 {\cal L}_Z \omega_2 = - \epsilon_2 \omega_3 = \epsilon_1 \epsilon_3 \omega_3 \;.
	\end{align*}
\end{proof}
Let $M'$ be a codimension one submanifold of $P$ transversal to $Z_1^P$,  that is $TP|_{M'}=TM'\oplus\mathbb{R} Z_1^P$. Let $pr_{TM'}$ denote the projection onto the first factor, i.e.\ the projection from $TP$ to $TM'$ along $Z_1^P$. On $M'$ we define
\[
	X := pr_{TM'} \circ X_P|_{M'} \;,
\]
and for any vector field $Y$ on $M$ we define the vector field on $M'$
\[
	Y' := pr_{TM'} \circ \tilde{Y}|_{TM'} \;.
\]
Since $pr_{TM'} Z_1^P = 0$ we have
\begin{equation}
	X = pr_{TM'} \circ X_P|_{M'}  = -\frac{1}{f_1} pr_{TM'} \circ \tilde{Z} |_{M'} =  -\frac{1}{f_1} Z' \;.
\end{equation}
We define the distributions 
\[
	{\cal D}'{}^{h} := \text{span}\{ Y' \;\; | \;\; Y \in \Gamma({\cal D}^h) \} \subset TM' \;,
\]
and
\[
	{\cal D}'{}^v := \text{span} \{ X, (J_1 Z)', (J_2 Z)', (J_3 Z)' \}
	= \text{span} \{ Z', (J_1 Z)', (J_2 Z)', (J_3 Z)' \} \subset TM' \;.
\]
We can then decompose the tangent space of $M'$ as
\[
	TM' = {\cal D}'{}^v \oplus {\cal D}'{}^h \;.
\]

\begin{proposition}
	An almost $\varepsilon$-hyper-complex structure  $(J_1', J_2', J_3')$ on $M'$ such that $J_\alpha'^2=\epsilon_\alpha Id_{TM'}$ and $J_1J_2=J_3$ is uniquely defined by
	\[
		J'_\alpha X = -\frac{1}{f_1} (J_\alpha Z)' \;,
		\qquad
		J_\alpha'(J_\beta Z)' = (J_\gamma Z)' \;,
	\]
	and
	\[
		J'_\alpha(Y') = (J_\alpha Y)' \qquad \text{for all} \;\;\; Y' \in \Gamma({\cal D}'{}^h) \;.	
	\]
\end{proposition}

\begin{proof}
	Since $J_\alpha$ preserves ${\cal D}^h$ it follows that $J'_\alpha$ preserves ${\cal D}'{}^h$. It is also clear that $J'_\alpha$ preserves ${\cal D}'{}^{v}$. The matrices representing $J_1'|_{{\cal D}'{}^{v}}, J_2'|_{{\cal D}'{}^{v}}, J_3'|_{{\cal D}'{}^{v}}$ are given by Eq.\ \eqref{eq:matrices}, and for all $Y' \in {\cal D}'{}^h$ we have
	\[
		J'_{\alpha_1} J'_{\alpha_2} Y' = J'_{\alpha_1} (J_{\alpha_2} Y)' = (J_{\alpha_1} J_{\alpha_2} Y)' \;.
	\]
Therefore $(J'_1, J'_2, J'_3)$ fulfil the $\varepsilon$-quaternionic algebra

\begin{equation*}
	J_1'{}^2 = \epsilon_1 Id_{TM'}\;,
	\qquad
	J_2'{}^2 = \epsilon_2 Id_{TM'}\;,
	\qquad
	J_3'{}^2 = \epsilon_3 Id_{TM'} \;,
	\qquad
	J_1' J_2' = J_3' \;.
\end{equation*}
\end{proof}

\begin{theorem}\label{mainThm}
	Let $(M,g,J_1,J_2,J_3)$ be an $\varepsilon$-hyper-K\"ahler manifold and let $f \in C^\infty(M)$ be a function on $M$ that fulfils the assumptions  stated above. Choose a rank-one principal bundle $P$ with connection $\eta$ and a submanifold $M' \subset P$ as above. Define
	\[
		Q := \text{span}\{ J_1', J_2', J_3' \} \;,
	\]
	with $J_1', J_2', J_3'$ defined as above. Then $(M',g',Q)$ with
	\begin{equation}	
		g' := \left. \frac{1}{2|f|} \left( g_P - \frac{2}{f} \left( (\theta_1^P)^2 - \epsilon_1 (\theta_0^P)^2 - \epsilon_2 (\theta_3^P)^2 - \epsilon_3 (\theta_2^P)^2 \right) \right) \right|_{M'}
	\end{equation}
is an $\varepsilon$-quaternionic K\"ahler manifold and $X$ is Killing with respect to $g'$.
\end{theorem}

\noindent {\it Proof for} $\dim_\mathbb{R}\!M>4$:	\\[0.1cm]
Substituting in \eqref{eq:gp} and then \eqref{eq:gdecomp} we find
\begin{align*}
	g' &= \left. \frac{1}{2|f|} \left(   \frac{2}{f_1} \eta^2 + \pi^*g - \frac{2}{f} \left( (\theta_1^P)^2 - \epsilon_1 (\theta_0^P)^2 - \epsilon_2 (\theta_3^P)^2 - \epsilon_3 (\theta_2^P)^2 \right)  \right) \right|_{M'} \\
	&=  \frac{1}{2|f|} \Bigg(   \frac{2}{f_1} \eta^2 + \frac{4}{\beta(Z)} \left( (\theta_1^P - \eta)^2 - \epsilon_1 (\theta_0^P)^2 - \epsilon_2 (\theta_3^P)^2 - \epsilon_3 (\theta_2^P)^2\right) + \pi^* \check{g} \\
	&\hspace{14em}- \frac{2}{f} \left( (\theta_1^P)^2 - \epsilon_1 (\theta_0^P)^2 - \epsilon_2 (\theta_3^P)^2 - \epsilon_3 (\theta_2^P)^2 \right)  \Bigg) \Bigg|_{M'} \;.
\end{align*}
We then use the fact that
\[
	\frac{4}{\beta(Z)} - \frac{2}{f}= \frac{4f_1}{f \beta(Z)} \;,
	\qquad
	\frac{4}{\beta(Z)} + \frac{2}{f_1}= \frac{4f}{f_1 \beta(Z)} \;,
\]
and
\[
	\frac{2}{f_1} \eta^2  + \frac{4}{\beta(Z)}  (\theta_1^P - \eta)^2- \frac{2}{f}  (\theta_1^P)^2 = \frac{4f_1}{f\beta(Z)} ( \theta_1^P - \frac{f}{f_1} \eta)^2 \;,
\]
to write this as
\begin{align*}
	g' &=  \frac{1}{2|f|} \Bigg(   \frac{4f_1}{f \beta(z)} \left( (\theta_1^P - \frac{f}{f_1} \eta)^2  - \epsilon_1 (\theta_0^P)^2 - \epsilon_2 (\theta_3^P)^2  - \epsilon_3 (\theta_2^P)^2\right) + \pi^* \check{g}  \Bigg) \Bigg|_{M'} \\
	&= \lambda \sigma \sigma_1 \left(  (\theta_1')^2 - \epsilon_1 (\theta_0')^2 - \epsilon_2 (\theta_3')^2 - \epsilon_3 (\theta_2')^2 \right) + \frac{1}{2|f|} \pi^* \check{g} \Big|_{M'} \;,
\end{align*}
where
\begin{align*}
	\theta_0' &:= \frac{1}{|f|}  \sqrt{\left| \frac{2f_1}{\beta(Z)} \right|} \theta_0^P \big|_{M'}\;, & 
	\theta_1' &:=\frac{1}{|f|}  \sqrt{\left| \frac{2f_1}{\beta(Z)} \right|} \left. \left(\theta_1^P -\frac{f}{f_1} \eta\right) \right|_{M'}\;, \\
	\theta_2' &:= \frac{1}{|f|}  \sqrt{\left| \frac{2f_1}{\beta(Z)} \right|} \theta_2^P \big|_{M'} \;, & 
	\theta_3' &:= \frac{1}{|f|}  \sqrt{\left| \frac{2f_1}{\beta(Z)} \right|} \theta_3^P \big|_{M'} \;.	
\end{align*}
and
\[
	\lambda := \text{sign} \beta(Z) \;, \qquad \sigma = \text{sign} f\;, \qquad \sigma_1 = \text{sign} f_1 \;.
\]
Since $Z_1^P$ lies in the kernel of $\theta_0^P, \theta_1^P - f/f_1 \eta, \theta_2^P, \theta_3^P$ the above splitting of $g'$ corresponds to the splitting $TM' = {\cal D}'{}^v \oplus {\cal D}'{}^h$ given previously. Therefore the first summand is non-degenerate on ${\cal D}'{}^v$ and has kernel ${\cal D}'{}^h$ whilst the second summand is non-degenerate on ${\cal D}'{}^h$ and has kernel ${\cal D}'{}^v$. Note that this already implies that $g'$ is non-degenerate.	In addition $g'$ is invariant under $X$.

We will show that
\begin{equation}
	\omega_\alpha' := -\epsilon_\alpha g'(J_\alpha \cdot, \cdot) = \frac{\sigma}{2}\epsilon_1 \epsilon_\alpha\, d \bar{\theta}_\alpha +\sigma\epsilon_2\,\bar{\theta}_\beta \wedge \bar{\theta}_\gamma  \;
	\label{eq:omegaprime}
\end{equation}
where we have defined the one-forms
\begin{equation}
	\bar{\theta}_\alpha := \frac{1}{f} \theta_\alpha^P \Big|_{M'} \;. \label{eqThetaBar}
\end{equation}
Differentiating gives
\begin{equation}
	d \omega'_\alpha = 2\epsilon_3 \left(\epsilon_\gamma \bar{\theta}_\beta \wedge \omega'_\gamma - \epsilon_\beta \bar{\theta}_\gamma \wedge \omega'_\beta \right) \;.\label{eqDefThetaBar}
\end{equation}
From these expressions it follows immediately that the fundamental four-form
\[
	\Omega_4 := \sum_{\alpha = 1,2,3} \epsilon_\alpha \, \omega'_\alpha \wedge \omega'_\alpha 
\]
is closed, and that the algebraic ideal generated by $(\omega'_1, \omega'_2, \omega'_3)$ in $\Omega^\ast(M')$ is a differential ideal.
In dimensions greater than eight the closure of the fundamental four-form is enough to show that the metric is para-quaternionic K\"ahler. In dimension eight the closure of the fundamental four-form along with the fact that the fundamental two-forms generate a differential ideal is enough to show that the metric is $\varepsilon$-quaternionic K\"ahler. Both statements were originally stated in \cite{S} for almost quaternionic pseudo-Hermitian manifolds. The proof in the appendix of \cite{S} is based on complex representation theory that does not depend on the respective real form. Hence both statements can be generalised to the para-quaternionic K\"ahler case (see \cite{Dancer:2004} for the first statement in the para-quaternionic case). This then completes the proof of Theorem \ref{mainThm}.

It is left to verify equation \eqref{eq:omegaprime}. Since $(J_1', J_2', J_3')$ agrees with $(J_1, J_2, J_3)$ on ${\cal D}'{}^h$ we have
\[
	\pi^* \left( - \epsilon_\alpha \check{g}\right)|_{M'}(J_\alpha' \cdot, \cdot) = \pi^* (- \epsilon_\alpha\check{g} (J_\alpha \cdot, \cdot))|_{M'} = \pi^* \check{\omega}_\alpha |_{M'} \;.
\]
On ${\cal D}'{}^v$ the vectors $X, J_1 X, J_2 X, J_3 X$ are pairwise orthogonal with respect to 
\[
	(\theta_1')^2 - \epsilon_1 (\theta_0')^2 - \epsilon_2 (\theta_3')^2 - \epsilon_3 (\theta_2')^2 \;,
\]
and fulfil
\[
	\theta_0'(J_1'X) = -\theta_1'(X) = \epsilon_2 \theta'_2(J_3'X) = -\epsilon_2 \theta_3'(J_2'X) = \frac{\lambda \sigma_1}{|f|} \sqrt{\left|\frac{\beta(Z)}{2f_1}\right|}  \neq 0 \;.
\]
We therefore have
\[
	J'_\alpha{}^* \theta'_0 = -\theta'_\alpha \;, 
	\qquad
	J_1'{}^* \theta'_2 = -\theta'_3 \;,
	\qquad
	J_2'{}^* \theta'_3 = \epsilon_2 \theta'_1 \;,
	\qquad
	J_3'{}^* \theta'_1 = \epsilon_1 \theta'_2 \;.
\]
Next, we note that  $\theta_\alpha^P = \pi^* \theta_\alpha$ for $\alpha = 0,2,3$ and
\begin{align*}
	\left. \left( \frac{2}{|f| \beta(Z)} \pi^* \theta_1 - \frac{\sigma}{f^2} \theta_1^P \right) \right|_{M'}
	&= \left( \frac{2}{|f|\beta(Z)} (\theta_1^P - \eta) - \frac{\sigma}{f^2} \theta_1^P \right) \\
	&= \left(\frac{1}{|f|} \frac{2f_1}{f\beta(Z)}(\theta_1^P - \frac{f}{f_1} \eta) \right) 
	= \lambda \sigma \sigma_1 \frac{1}{|f|} \sqrt{\left| \frac{2f_1}{\beta(Z)} \right|} \theta_1' \;.
\end{align*}
From Lemma \ref{lem:theta} we have
\[
	d\bar{\theta}_\alpha = \frac{1}{f} \epsilon_1 \epsilon_\alpha \pi^* \omega_\alpha - \frac{1}{f^2} df \wedge \theta_\alpha^P\;.
\]
Putting everything together we find
\begin{align*}
	\omega_\alpha' &= \lambda \sigma \sigma_1(\epsilon_1 \epsilon_\alpha \theta'_0 \wedge \theta_\alpha' - \epsilon_2 \theta'_\beta \wedge \theta'_\gamma)
	+ \frac{1}{2|f|} \pi^* \check{\omega}_\alpha |_{M'} \\
	&= \left. \left( \frac{1}{2|f|} \pi^* \check{\omega}_\alpha  + \frac{2}{|f|\beta(Z)}  \pi^* (\epsilon_1 \epsilon_\alpha \theta_0 \wedge \theta_\alpha - \epsilon_2 \theta_\beta \wedge \theta_\gamma) - \frac{\sigma}{f^2} (\epsilon_1 \epsilon_\alpha\theta_0^P \wedge \theta_\alpha^P - \epsilon_2 \theta_\beta^P \wedge \theta_\gamma^P)  \right)\right|_{M'} \\
	&= \left. \left( \frac{1}{2|f|} \pi^* \omega_\alpha - \epsilon_1 \epsilon_\alpha \frac{\sigma}{2f^2} df \wedge \theta_\alpha^P + \epsilon_2 \frac{\sigma}{f^2} \theta_\beta^P \wedge \theta_\gamma^P \right)\right|_{M'} \\
	&= \frac{\sigma}{2}(\epsilon_1 \epsilon_\alpha d\bar{\theta}_\alpha + \epsilon_2 2\bar{\theta}_\beta \wedge \bar{\theta}_\gamma) \;.
\end{align*}
This verifies Eq.\ \eqref{eq:omegaprime} and ends the proof in dimension greater than four.
\qed
~\\

\noindent {\it Proof for }$\dim_\mathbb{R}\!M=4$: \\[0.1cm]
The proof in dimension four relies on the fact that on any submanifold $M'$ of an $\varepsilon$-quaternionic K\"ahler manifold $(\tilde{M},\tilde{g},\tilde{Q})$ such that $TM'$ is $\tilde{Q}$-invariant, \[(Q':=\tilde{Q}\big|_{M'},g':=\tilde{g}\big|_{M'})\] defines an $\varepsilon$-quaternionic K\"ahler structure (see \cite[Prop.\ 8]{M} and references therein). This idea is taken from \cite[Cor. 4.2.]{MS2}.

Assume that $\dim_\mathbb{R}\!M=4$. Let $M_0:=\mathbb{R}^4$ be endowed with standard real coordinates $(x,y,u,v)$. Using $z:=x+i_{\epsilon_1}y$ and $w:=u-i_{\epsilon_1}v$, we define an $\varepsilon$-hyper-K\"ahler structure $(g_0,\,J_1^0,\,J_2^0,\,J_3^0)$ by \[g_0:=dzd\bar{z}-\epsilon_2 dwd\bar{w},\qquad \omega^0_+:=\omega_2^0+i_{\epsilon_1}\omega_3^0=dz\wedge dw.\]
Let $f^0:=\epsilon_1\epsilon_2 w\bar{w}\in C^\infty(M_0)$. This defines a $J_1^0$-holomorphic vector field \[Z^0:=-(\omega_1^0)^{-1}(df)=-2\epsilon_1 i_{\epsilon_1}(w\partial_w-\bar{w}\partial_{\bar{w}})\] that fulfils $\mathcal{L}_Z \omega^0_+=-2\epsilon_1i_{\epsilon_1}\omega_+^0$ and, hence, $\mathcal{L}_{Z^0}J_2^0=2\epsilon_1 J_3^0$. 
The one-form $\eta^{M_0}_0:=\frac{1}{2}\mathrm{Im}(\bar{z}dz+\epsilon_2\bar{w}dw)$ fulfils $d\eta^{M_0}_0=\omega_1^0-\frac{1}{2}d(\iota_{Z^0}g_0)$ and we have $f_1^0:=f^0-\frac{1}{2}g_0(Z^0,\,Z^0)=-\epsilon_1\epsilon_2 w\bar{w}$.

Consider $(\tilde{M}:=M\times \mathbb{R}^4,\,\tilde{g}:=g+g_0,\,\tilde{f}:=f+f^0)$ together with the $\varepsilon$-hyper-complex structure $(\tilde{J}_1,\,\tilde{J}_2,\,\tilde{J}_3)$.
Let $\tilde{U}\subset \tilde{M}$ be a neighbourhood of $M=M\times\{0\}\subset\tilde{M}$ such that the signs of $\tilde{f}$, $\tilde{f}_1:=f_1+f_1^0$ and $\tilde{f}-\tilde{f}_1$ restricted to $\tilde{U}$ are constant. Then the restriction of the above data from $\tilde{M}$ to $\tilde{U}$ fulfils the assumptions of the $\varepsilon$-HK/QK correspondence. The restriction of $P\times \mathbb{R}^4$ defines a rank-one principal bundle $\tilde{P}$ over $\tilde{U}$ with connection $\tilde{\eta}=(\eta+\eta_0^{M_0})\big|_{\tilde{P}}$.
The $\varepsilon$-HK/QK correspondence with the choices $(\tilde{P},\,\tilde{\eta},\,\tilde{M}':=M'\times \mathbb{R}^4)$
then defines an $\varepsilon$-quaternionic K\"ahler structure $(\tilde{g}',\,\tilde{Q})$ on the eight-dimensional manifold $\tilde{M}'$. The submanifold $M'=M'\times \{0\}\subset \tilde{M}'$ has a $\tilde{Q}$-invariant tangent bundle and, hence, $(M',\,\tilde{g}'\big|_{M'},\,\tilde{Q}\big|_{M'})$ is $\varepsilon$-quaternionic K\"ahler. The one-forms $\tilde{\bar{\theta}}_\alpha$ on $\tilde{M}'$ obtained from the $\varepsilon$-HK/QK-correspondence (see Eq.\ \eqref{eqThetaBar}) restrict to the corresponding one-forms $\bar{\theta}_\alpha$ on $M'$. The latter one-forms define the $\varepsilon$-quaternionic K\"ahler structure $(g',Q)$ on $M'$ obtained from the $\varepsilon$-HK/QK correspondence by Eq.\ \eqref{eq:omegaprime}, which in particular shows that $(\tilde{g}'\big|_{M'},\,\tilde{Q}\big|_{M'})=(g',\,Q)$.
\qed

~

\begin{remark}
On any $\varepsilon$-quaternionic K\"ahler manifold one can define three one-forms $(\bar{\theta}_\alpha)_{\alpha = 1,2,3}$ by
\label{rem:1}
\begin{equation}
	\nabla_\cdot J_\alpha' = 2\epsilon_3 \epsilon_\alpha\left( \bar{\theta}_\beta(\cdot) \, J'_\gamma - \bar{\theta}_\gamma(\cdot)\, J'_\beta \right) \;,
	\label{eq:NJ}
\end{equation}
which is equivalent to Eq.\ \eqref{eqDefThetaBar}. Then the fundamental two-forms fulfil \cite[Prop.\ 5]{AC2}\footnote{Compared to \cite{AC2} we have $\omega'_\alpha=\rho'_\alpha\!^{\mathrm{[AC2]}}$ and $\bar{\theta}_\alpha=-\frac{\epsilon_3\epsilon_\alpha}{2}\omega_\alpha\!^{\mathrm{[AC2]}}$.}
\begin{equation}
\frac{\nu}{2}\omega'_\alpha=-\epsilon_\alpha d\bar{\theta}_\alpha+2\epsilon_3\bar{\theta}_\beta\wedge\bar{\theta}_\gamma\;,\label{eqAleks}\end{equation}
where $\nu:=\frac{scal}{4n(n+2)}$ $(\mathrm{dim}_{\mathbb{R}}M'=4n)$ is the reduced scalar curvature. Comparing Eqs.\ \eqref{eq:omegaprime} and \eqref{eqAleks} shows that the reduced scalar curvature of $g'$ is $\nu=-\epsilon_1 4\sigma$.
\end{remark}

~

\begin{theorem}
	Let $(M',g',Q)$ be an $\varepsilon$-quaternionic K\"ahler manifold in the image of the $\varepsilon$-HK/QK correspondence as described above. The globally defined almost $\epsilon_1$-complex structure $J_1' \in \Gamma(Q)$ is integrable.
	\label{thm:cx}
\end{theorem}

\begin{proof}

Let $a \in {C}^\infty(P)$ such that $X = (X_P - aZ_1P)|_{M'}\in {\mathfrak X}(M').$ 
We will identify the $\varepsilon$-quaternionic K\"ahler moment map associated with $X$ and use it to prove that $J_1'$ is integrable.

The Killing vector field $X$ satisfies 
\[
	(\iota_X \bar{\theta})_{\alpha =1,2,3} = \left((f'{}^{-1} - a') , 0, 0\right) \;,
	\qquad
	(\iota_X d \bar{\theta})_{\alpha = 1,2,3} =\left(  f'{}^{-2} df' ,  -\epsilon_2a' 2 \bar{\theta}_3, \epsilon_2 a' 2 \bar{\theta}_2 \right) \;,
\]
where we have defined $f' = f|_{M'}$ and $a' = a|_{M'}$. The Lie derivative of $\bar{\theta}_\alpha$ is therefore
\[
	{\cal L}_X \bar{\theta}_\alpha = \left(-da', -\epsilon_2 a' 2 \bar{\theta}_3, \epsilon_2 a'2 \bar{\theta}_2\right) \;.
\]
From \eqref{eqAleks} it  follows that
\[
	{\cal L}_X \omega'_\alpha = \frac{2}{\nu} \left[ -\epsilon_\alpha {\cal L}_X d \bar{\theta}_\alpha + \epsilon_3 2{\cal L}_X \bar{\theta}_\beta \wedge \bar{\theta}_\gamma + \epsilon_32  \bar{\theta}_\beta \wedge {\cal L}_X\bar{\theta}_\gamma \right]\;,
\]
which is calculated to be
\begin{equation}
	({\cal L}_X \omega'_\alpha)_{\alpha = 1,2,3} = \big( 0, -\epsilon_3 2a' \omega'_3, -2a' \omega'_2\big) \;.
	\label{eq:LXO1}
\end{equation}

The following theorem is proved in  \cite[Thm.\ 2.4]{GL} in the quaternionic case and \cite[Thm 5.2]{V} in the para-quaternionic case.
\begin{theorem}
	Let $X \in {\mathfrak X}(M') $ be a Killing vector field on an $\varepsilon$-quaternionic K\"ahler manifold $(M',g',Q)$. There exists a unique section $\mu^X \in \Gamma(Q)$ on an open subset $U \subset M'$ such that
	\[
		\nabla_\cdot \mu^X \big|_U = \omega'_\alpha(X, \cdot) J'_\alpha \;.
	\] 	\label{thm:mu}
\end{theorem}
\begin{proposition}
	The Lie derivative of $\omega'_\alpha$ may be written in terms of $\mu^X =: \sum_{\alpha = 1}^3 \mu^X_\alpha J'_\alpha$ as
\begin{equation}
	{\cal L}_X \omega'_\alpha = \epsilon_3 (\epsilon_\alpha \nu\mu^X_\beta + \epsilon_\gamma 2 \bar{\theta}_\beta(X))\omega'_\gamma -  \epsilon_3  (\epsilon_\alpha \nu \mu^X_\gamma + \epsilon_\beta 2\bar{\theta}_\gamma(X))\omega'_\beta\;. \label{eq:LXO2}
\end{equation}
\end{proposition}
\begin{proof}
	From \eqref{eq:NJ} it follows that
	\begin{equation}
		\nabla_\cdot \omega'_\alpha= \epsilon_32 (\bar{\theta}_\beta (\cdot)  \epsilon_\gamma \omega'_\gamma - \bar{\theta}_\gamma(\cdot) \epsilon_\beta \omega'_\beta ) \;.
		\label{eq:NO}
	\end{equation}
\noindent Expanding  $\mu^X = \sum_{\alpha = 1}^3 \mu^X_\alpha J'_\alpha$ and making use of \eqref{eq:NJ} we find
	\begin{align*}
		\nabla_\cdot (\mu^X_\alpha J'_\alpha) 
		&= d\mu^X_\alpha J'_\alpha +  \epsilon_3 \epsilon_\alpha \mu^X_\alpha  2(\bar{\theta}_\beta(\cdot)J'_\gamma - \bar{\theta}_\gamma (\cdot) J'_\beta) \;,
	\end{align*}
and since $(J'_1, J'_2, J'_3)$ are linearly independent it follows from Theorem \ref{thm:mu} that 
	\begin{equation}
		d\mu^X_\alpha + \epsilon_3 \epsilon_\beta 2\mu_\beta^X \bar{\theta}_\gamma - \epsilon_3 \epsilon_\gamma 2\mu_\gamma^X \bar{\theta}_\beta = \iota_X \omega'_\alpha \;.
		\label{eq:dmu}
	\end{equation}
	Equations \eqref{eq:NO} and \eqref{eq:dmu} together with the fact that $\epsilon_\alpha \epsilon_\beta = -\epsilon_\gamma$ imply
	\begin{align}
		 \text{alt}(\nabla_\cdot \omega'_\alpha)(X,\cdot) 
		 &=  \epsilon_3  (\bar{\theta}_\beta \wedge \epsilon_\gamma d\mu_\gamma^X - \bar{\theta}_\gamma \wedge \epsilon_\beta d \mu^X_\beta) \notag\\
		&\hspace{2em} + \epsilon_\alpha 2  \mu^X_\beta \bar{\theta}_\beta \wedge \bar{\theta}_\alpha
		 +\epsilon_\alpha 2 \mu^X_\gamma \bar{\theta}_\gamma \wedge \bar{\theta}_\alpha  \;, \label{eq:OX}
	\end{align}
where alt is the anti-symmetrisation operator, i.e.\ 
\[
 \text{alt}(\nabla_Y \omega'_\alpha)(X,Z) =  \frac12((\nabla_Y \omega'_\alpha)(X,Z) - (\nabla_Z \omega'_\alpha)(X,Y)) \;.
\]
The Lie derivative of $\omega'_\alpha$ may be written as
	\begin{align*}
		{\cal L}_X \omega'_\alpha 
		&= \iota_X d\omega'_\alpha + d\iota_X \omega' \\
		&= \nabla_X \omega'_\alpha - 2\,\text{alt}(\nabla_\cdot \omega'_\alpha)(X,\cdot) + \epsilon_3 2d(\epsilon_\beta \mu^X_\beta \bar{\theta}_\gamma - \epsilon_\gamma  \mu^X_\gamma \bar{\theta}_\beta)  \;.
	\end{align*}
Substituting \eqref{eq:NO}, \eqref{eq:dmu} and \eqref{eq:OX} into the above expressions produces the desired result.
\end{proof}

Comparing the two expressions \eqref{eq:LXO1} and \eqref{eq:LXO2} for the Lie derivative of $\omega'_\alpha$ we find 
\begin{equation}
	\mu^X_\alpha = \left( -\frac{1}{2|f'|} , 0, 0\right) \qquad \Rightarrow \qquad \mu^X = -\frac{1}{2|f'|} J'_1 \;.
	\label{eq:muX}
\end{equation}

	The expression \eqref{eq:muX} is already enough to prove that $J'_1$ is an integrable $\epsilon_1$-complex structure. To show this we will adapt the proof of the statement in the case of an almost-complex structure on a quaternionic K\"ahler manifold given in \cite[Prop.\ 3.3]{B}.	Using \eqref{eq:dmu} we have $\bar{\theta}_2 = \epsilon_2 |f'| \iota_X \omega'_3$ and $\bar{\theta}_3 = -\epsilon_2 |f'| \iota_X \omega'_2$, hence
	\begin{equation}
		\nabla J'_1 = -2|f'|(\iota_X \omega'_2 \otimes J'_2 - \iota_X \omega'_3 \otimes J'_3) \;.
		\label{eq:NJ1}
	\end{equation}
	Fix a point $x \in M'$. Consider a one-form $\alpha$ that satisfies  $(\nabla \alpha)_x = 0$. Define the $\epsilon_1$-complex one-form $A = \alpha + \epsilon_1 i_{\epsilon_1} J'_1{}^* \alpha$, where $i_{\epsilon_1}$ is the $\epsilon_1$-complex unit satisfying $i_{\epsilon_1}^2 = \epsilon_1$ and $\bar{i}_{\epsilon_1} = -i_{\epsilon_1}$. The one-form $A$ is $J'_1$-holomorphic, that is $J'_1{}^* A = i_{\epsilon_1} A $.
	Using \eqref{eq:NJ1} we have
	\[
		(\nabla A)_x =  -\epsilon_1 i_{\epsilon_1}( (\nabla J'_1{}^*) \alpha)_x = \epsilon_1 i_{\epsilon_1} 2|f'| \left( \iota_X  \omega'_2  \otimes J'_2{}^* \alpha + \iota_X  \omega'_3  \otimes J'_3{}^* \alpha \right) \big|_x\;,
	\]
	and therefore 
	\[
		(d A)_x = \epsilon_1 i_{\epsilon_1} 2|f'| \left( \iota_X  \omega'_2  \wedge J'_2{}^* \alpha + \iota_X  \omega'_3  \wedge J'_3{}^* \alpha \right)  \big|_x\;.
	\]

	We now define
	\[
		\beta_{i} := 2|f'| \iota_X \omega'_{i} \;,
		\qquad
		\gamma_{i}^\alpha := J'_{i}{}^*\alpha \;,
		\qquad
		i = 2,3 \;,
	\]
	which satisfy
	\[
		J'_1{}^* \beta_{2} = \epsilon_1 \beta_3 \;, \qquad J'_1{}^*\beta_3 = \beta_2 \;,
		\qquad
		J'_1{}^* \gamma^\alpha_{2} = -\gamma^\alpha_3 \;, \qquad J'_1{}^* \gamma^\alpha_3 = -\epsilon_1 \gamma^\alpha_2 \;.
	\]
	We may then write
	\begin{align*}
		(d A)_x 
		 &= -\frac{\epsilon_1 i_{\epsilon_1}}{2} \left( B \wedge \bar{C}^\alpha
		 + \bar{B} \wedge C^\alpha \right)_x \;,
	\end{align*}
	where $B := \beta_2  + \epsilon_1 i_{\epsilon_1} J'_1{}^* \beta_2$ and $C^\alpha := \gamma^\alpha_2 + \epsilon_1 i_{\epsilon_1} J'_1{}^* \gamma^\alpha_2$\;. Since $B$ and $C^\alpha$ are $J'_1$-holomorphic this shows that $J'_1$ is integrable by the Newlander--Nirenberg theorem in the complex case and Frobenius' theorem in the para-complex case. 
This completes the proof of Theorem \ref{thm:cx}.
\end{proof}

\section{One-parameter deformations of the local temporal and Euclidean $c$-maps}

In this section we will consider three important examples of the $\varepsilon$-HK/QK correspondence. They are related to constructions in the physics literature known as the local (or supergravity) spatial, temporal and Euclidean $c$-maps.

Let $\epsilon_1,\epsilon_2\in\{-1,1\}$ and $\epsilon_3=-\epsilon_1\epsilon_2$ as in Section 1. 
We start with a conical affine special $\epsilon_1$-K\"ahler (CAS$\epsilon_1$K) manifold $M$ and consider the $\varepsilon$-hyper-K\"ahler manifold $N$ obtained from the global (or rigid) spatial $c$-map for $(\epsilon_1, \epsilon_2) = (-1, -1)$, temporal $c$-map for $(\epsilon_1, \epsilon_2) = (-1, 1)$ or Euclidean $c$-map for $(\epsilon_1, \epsilon_2) = (1, \pm1)$. 
One can identify $N$ with the cotangent bundle of $M$.
Using a natural vector field fulfilling the assumptions of Section \ref{secPHKQK}, we apply the $\varepsilon$-HK/QK correspondence to $N$ and obtain a family of $\varepsilon$-quaternionic K\"ahler manifolds $\bar{N}_c$ that depends on a  real parameter $c$. 
For $c=0$ the result agrees with the $\varepsilon$-quaternionic K\"ahler manifold obtained from the projective special $\epsilon_1$-K\"ahler (PS$\epsilon_1$K) manifold $\bar{M}$ underlying $M$ via the local spatial, temporal or Euclidean $c$-map. 
This construction is summarised in the following diagram:
\[
	\xymatrix
	{
		\stackrel[\text{CAS$\epsilon_1$K}, \; 2n + 2]{}{M} \ar[dd]  \ar@{|->}[rrr(0.95)]^{\text{global $c$-map}}& &  &  \hspace{3.5em} \stackrel[\text{$\varepsilon$-HK}, \; 4n + 4]{}{N = T^* M} \ar@{|->}@<3.8ex>[dd]^{\text{$\varepsilon$-HK/QK}}  \\
		\\
		\stackrel[\text{PS$\epsilon_1$K}, \; 2n]{}{\bar{M}} \ar@{|->}[rrr]^{\text{local $c$-map}}_{(\text{undeformed})} & &    &  \bar{N}_{c = 0} \in  \stackrel[\text{$\varepsilon$-QK}, \; 4n + 4]{}{\bar{N}_c}
	}
\]
In the specific case of the local spatial $c$-map there is a known one-parameter deformation called the \emph{one-loop deformation} of the local spatial $c$-map \cite{RoblesLlana:2006ez}. It was shown in \cite{ACDM} that the family of target manifold $\bar{N}_c$ in the above example of the $\varepsilon$-HK/QK correspondence with $(\epsilon_1, \epsilon_2) = (-1, -1)$ (i.e.\ the HK/QK correspondence) and $c \neq 0$ is in exact agreement with the family of target manifolds of the of the one-loop deformation of the local spatial $c$-map.
On the other hand, no deformations of the local temporal or Euclidean $c$-maps have appeared in the literature.
We will show that the above example of the $\varepsilon$-HK/QK with $(\epsilon_1, \epsilon_2) = (-1, 1)$ or $(1, \pm1)$ and $c \neq 0$ results in two new families of para-quaternionic K\"ahler manifolds $\bar{N}_c$ that can be understood as one-parameter deformations of the local temporal and Euclidean $c$-map target manifolds.

Let $(M, g_M, J, \nabla,\xi)$ be a conical affine special $\epsilon_1$-K\"ahler manifold \cite{CM} of dimension $\mathrm{dim}_{\mathbb{R}} M=2(n+1)$. We assume that $g(\xi,\xi)>0$ and that if $\epsilon_1=-1$ then $g_M\big|_{\{\xi,J\xi\}^\perp\subset TM}<0$. 
Let $X=(X^I)=(X^0,\ldots,X^n): U\subset M\stackrel{\sim}{\to} \tilde{U}\subset \mathbb{C}^{n+1}$ be a set of conical special $\epsilon_1$-holomorphic coordinates such that the geometric data on the domain $U\subset M$ is encoded in an $\epsilon_1$-holomorphic function $F:\tilde{U}\to\mathbb{C}$ that is homogeneous of degree 2. The metric may then be written as\footnote{Note that apart from interchanging $X$ and $z$, we use conventions in this section that agree with \cite{ACDM} for  $(\epsilon_1, \epsilon_2) = (-1, -1)$. Compared to \cite{Cortes:2015wca} $N_{IJ}$ is defined in terms of the $\epsilon_1$-holomorphic prepotential $F$ with an extra minus sign. The same holds true for the definition of the projective special $\epsilon_1$-K\"ahler metric $g_{\bar{M}}$ in terms of the conical affine special $\epsilon_1$-K\"ahler metric $g_M$.
}
\begin{equation}
g_M=N_{IJ}dX^Id\bar{X}^J \;,
\qquad 
N_{IJ}(X,\bar{X}):=i_{\epsilon_1}(\bar{F}_{IJ}(\bar{X})-F_{IJ}(X))=-2\epsilon_1 \mathrm{Im}\,F_{IJ}(X) \;,
\end{equation}
and the Euler vector field as $\xi=X^I\partial_{X^I}+\bar{X}^I\partial_{\bar{X}^I}$, where $F_{IJ}(X):=\frac{\partial^2F(X)}{\partial X^I\partial X^J}$ for $I,J=0,\ldots,n$. The $\epsilon_1$-K\"ahler potential for $g_M$ is given by $r^2=g_M(\xi,\xi)=X^IN_{IJ}(X,\bar{X})\bar{X}^J$.

We will assume that $\xi$ and $J\xi$ induce free $\mathbb{R}^{>0}$- and $A^{J\xi}$-actions, respectively, where\footnote{The unit para-complex numbers have four connected components. Here the subscript $_0$ denotes the connected component of $1\in C$.}
\begin{equation}
	A^{J\xi} :=
	\begin{cases}
		\{ z = x + iy \in {\mathbbm C} \; :\; |z|^2 = x^2 + y^2 = 1 \} \simeq S^1 & \text{if} \;\;\; \epsilon_1 = -1 \\
		\{ z = x + ey \in {C} \; :\; |z|^2 = x^2 - y^2 =  1 \}_{_0}  \simeq {\mathbbm R}   & \text{if}  \;\;\; \epsilon_1 = +1 \;.
	\end{cases}
\end{equation}
Notice that $\mathbb{R}^{>0} \times A^{J\xi} \simeq {\mathbbm C}^*$ if $\epsilon_1 = -1$, and $\mathbb{R}^{>0} \times A^{J\xi} \simeq {C}^*_0$ if $\epsilon_1 = 1$.
Let \begin{equation}
\bar{\pi}:= M\to \bar{M}:=M/(\mathbb{R}^{>0} \times A^{J\xi}) \;,
\end{equation}
and define $r:=\sqrt{g_M(\xi,\xi)}$. Let $(\bar{M},-g_{\bar{M}},J_{\bar{M}})$ be the $\epsilon_1$-K\"ahler manifold obtained from the $\epsilon_1$-K\"ahler quotient with level set $\{r=1\}\subset M$. Then $(\bar{M},g_{\bar{M}},J_{\bar{M}})$ is a projective special $\epsilon_1$-K\"ahler manifold \cite{CM} that is positive definite if $\epsilon_1=-1$ and $g_M$ has complex inverse-Lorentz signature. Note that
\begin{equation}
g_M=dr^2-\epsilon_1 r^2 \tilde{\eta}^2-r^2\bar{\pi}^\ast g_{\bar{M}}\;,
\end{equation}
where
\begin{equation}
\tilde{\eta}:=\frac{1}{r^2}g_M(J\xi,\cdot)=-\frac{\epsilon_1}{r^2}\omega_1(\xi,\cdot)=d^c\,\mathrm{log}\,r=i_{\epsilon_1}(\bar{\partial}-\partial)\,\mathrm{log}\,r\;.\label{eqCASKmetricDecomp}
\end{equation}
Let us assume that $X^0\bar{X}^0>0$ and $\mathrm{Re}\,X^0>0$. Then the $(\mathbb{R}^{>0}\times A^{J\xi})$-invariant functions $z^\mu:=\frac{X^\mu}{X^0}$, $\mu=1,\ldots,n$, define a local $\epsilon_1$-holomorphic coordinate system on $\bar{M}$. The $\epsilon_1$-K\"ahler potential for $g_{\bar{M}}$ is \begin{equation}
\mathcal{K}:=-\mathrm{log} \left( \,z^IN_{IJ}(z,\bar{z})\bar{z}^J \right) \;.
\end{equation}

The spatial ($(\epsilon_1, \epsilon_2) = (-1, -1)$), temporal ($(\epsilon_1, \epsilon_2) = (-1, 1)$) and Euclidean ($(\epsilon_1, \epsilon_2) = (1, \pm1)$) global $c$-map assigns an $\varepsilon$-hyper-K\"ahler manifold to any affine special $\epsilon_1$-K\"ahler manifold \cite{CMMS}. We will now review this construction. First of all, note that the real coordinates \begin{equation}
(q^a)_{a=1,\ldots,2n+2}:=(x^I,y_J)_{I,J=0,\ldots,n}:=(\mathrm{Re}\,X^I,\mathrm{Re}\,F_J(X))_{I,J=0,\ldots,n}
\end{equation} on $M$ are $\nabla$-affine and fulfil \begin{equation}
\omega_M:=-\epsilon_1
g_M(J\cdot,\cdot)=-2 dx^I\wedge dy_I.
\end{equation}
With respect to the coordinates $(q^a)$, the function $H:=\frac{1}{2}X^IN_{IJ}\bar{X}^J$ on $M$ is a Hesse potential, i.e.\ $g_M=H_{ab}dq^adq^b$ where $H_{ab} := \frac{\partial^2H}{\partial q^a \partial q^b}$. 
The matrix-valued function $(H_{ab})_{a,b=1,\ldots,2n+2}$ and its inverse can be calculated to be
\begin{equation}
(H_{ab})=\begin{pmatrix}N-\epsilon_1RN^{-1}R & \epsilon_1 2RN^{-1}\\\epsilon_12N^{-1}R&-\epsilon_14N^{-1}\end{pmatrix},
\quad (H^{ab})=\begin{pmatrix}N^{-1} & \frac{1}{2}N^{-1}R\\\frac{1}{2}RN^{-1}&\frac{1}{4}(-\epsilon_1N+RN^{-1}R)\end{pmatrix},
\end{equation}
where $R_{IJ}:=2\mathrm{Re}\,F_{IJ}(X)$, i.e.\ $F_{IJ}=\frac{1}{2}(R_{IJ}-\epsilon_1i_{\epsilon_1}N_{IJ})$.
We consider the cotangent bundle $N:=T^\ast M$ and introduce real functions $(p_a):=(\tilde{\zeta}_I,\zeta^J)$ on $N$ such that together with the pullback of $(q^a)$ to $N$, they form a system of canonical coordinates, i.e.\ $p_a dq^a\in T^\ast M\mapsto (q^a,p_b)$.
The $\varepsilon$-hyper-K\"ahler structure on $N$ defined in \cite{CMMS} is given by
\begin{align}
g&=H_{ab}dq^adq^b+\epsilon_1\epsilon_2 H^{ab}dp_adp_b\nonumber\\
\omega_1&=-\Omega_{ab}dq^a\wedge dq^b+\frac{\epsilon_2}{4}\Omega^{ab}dp_a\wedge dp_b=-2dx^I\wedge dy_I-\frac{\epsilon_2}{2}d\tilde{\zeta}_I\wedge d\zeta^I\nonumber\\
\omega_2&=\epsilon_1 2\Omega_{ac}H^{cb}dq^a\wedge dp_b\nonumber\\
\omega_3&=dq^a\wedge dp_a=dx^I\wedge d\tilde{\zeta}_I+dy_I\wedge d\zeta^I,
\end{align}
where $(\Omega_{ab})=-(\Omega^{ab})=\begin{pmatrix}0& {\mathbbm 1}_{n+1}\\-{\mathbbm 1}_{n+1}&0
\end{pmatrix}$.

To make contact with the formulation of the local spatial, temporal and Euclidean $c$-map given in \cite{Cortes:2015wca} later in this section, we consider the tangent bundle $\tilde{N}:=TM$ of $M$ and introduce real functions $(\hat{q}^a)$ on $\tilde{N}$ by $\hat{q}^a \frac{\partial}{\partial q^a}\in T M\mapsto (q^a,\hat{q}^b)$. We identify $N=T^\ast M$ with $\tilde{N}=TM$ using the $\epsilon_1$-K\"ahler form $\omega_M: TM\to T^\ast M,~v\mapsto \omega_M(v,\cdot)$. After this identification, we have
\begin{equation}
\hat{q}^a=\frac{1}{2}\Omega^{ab}p_b.
\end{equation}
The the $\varepsilon$-hyper-K\"ahler structure on $\tilde{N}$ is given as follows (note that $\Omega^{ab} H_{bc}\Omega^{cd} = \epsilon_1 4 H^{ad}$):
\begin{align}
	g &= H_{ab} dq^a dq^b - \epsilon_2 H_{ab} d\hat{q}^a d\hat{q}^b \;,\nonumber\\
	\nonumber
		\omega_1 &= -\Omega_{ab} dq^a \wedge dq^b - \epsilon_2  \Omega_{ab} d\hat{q}^a \wedge d\hat{q}^b \;, \\\nonumber
		\omega_2 &= -H_{ab} d\hat{q}^a \wedge d{q}^b \;, \\
		\omega_3 &= 2 \Omega_{ab} d{q}^a \wedge d\hat{q}^b \;. 
\end{align}
Note that the complex structures are given by
\begin{align}\nonumber
		J_1 &= J^{a}_{\;\;b}  \frac{\partial}{\partial q^a} \otimes dq^b  - J^{a}_{\;\;b}  \frac{\partial}{\partial \hat{q}^a} \otimes d\hat{q}^b \;, \\
		J_2 &=  \frac{\partial}{\partial \hat{q}^a} \otimes dq^a  + \epsilon_2 \frac{\partial}{\partial q^a} \otimes d\hat{q}^a  \;, \nonumber\\ 
		J_3 &= \epsilon_2 J^{a}_{\;\;b}  \frac{\partial}{\partial q^a} \otimes d\hat{q}^b - J^{a}_{\;\;b}  \frac{\partial}{\partial \hat{q}^a} \otimes d{q}^b \;,  
	\end{align}
	where $J^{a}_{\;\;b}=-\tfrac12\Omega^{ac} H_{cb}$.

Consider the lift of the vector field $-\epsilon_1 2J\xi$ on $M$ to a vector field $Z$ on $\tilde{N}=TM$ such that $Z(\hat{q}^a)=0$, i.e.\
\begin{equation}
Z := -\epsilon_1 H_a \Omega^{ab} \frac{\partial}{\partial q^b} \in \mathfrak{X}(\tilde{N}) \;.
\end{equation}
The vector field $Z$ is a $J_1$-holomorphic Killing vector field such that ${\cal L}_Z J_2 = \epsilon_1 2J_3$. Since
\[
	d(-\epsilon_1 2H) = - \omega_1 (Z, \cdot)\;,
\]
the function $f = -\epsilon_1 (2H - c)$ fulfils $df=-\omega_1(Z,\cdot)$ for any $c\in \mathbb{R}$. We may therefore apply the $\varepsilon$-HK/QK correspondence to $\tilde{N}$ endowed with the above $\varepsilon$-hyper-K\"ahler structure. 
We begin by calculating
	\begin{align*}
		\beta = g(Z,\cdot)=- 4 q^a \Omega_{ab} dq^b \;,
		\quad
		\beta(Z) = -\epsilon_1 8H \;,
		\quad
		f_1 = f-\frac{1}{2}\beta(Z)=\epsilon_1 ( 2H + c)\;,
	\end{align*}
\begin{equation}
	\omega_1 - \frac12 d\beta = -\epsilon_2 \Omega_{ab} d\hat{q}^a \wedge d\hat{q}^b + \Omega_{ab} d{q}^a \wedge d{q}^b  \;.
\end{equation}
Let $\pi: P:= {\mathbbm R} \times \tilde{N}\to\tilde{N}$ be the trivial ${\mathbbm R}$-bundle over $\tilde{N}$ and let $s$ denote the standard coordinate on the fibre of $P$ such that $X_P=\frac{\partial}{\partial s}$. We define $\tilde{\phi}:=-\epsilon_2 2s$.
Using the above information we see that the connection one-form
\begin{equation}
\eta := -\epsilon_2 \tfrac12 d\tilde{\phi} - \epsilon_2 \hat{q}^a \Omega_{ab} d\hat{q}^b + {q}^a \Omega_{ab} d{q}^b \;
\end{equation}
has curvature $d\eta=\pi^\ast(\omega_1-\frac{1}{2}d\beta)$ and satisfies $\eta(X_P) = 1$. 
The one-forms defined in Eq.\ \eqref{eqDefThetaP} are calculated to be:
\begin{align}
	\theta_0^P &=\frac{1}{2}df= -\epsilon_1 dH \;, \notag \\
	\theta_1^P &=\eta+\frac{1}{2}\beta= -\epsilon_2 \tfrac12 d\tilde{\phi} - \epsilon_2 \hat{q}^a \Omega_{ab} d\hat{q}^b - {q}^a \Omega_{ab} d{q}^b \;, \notag \\
	\theta_2^P &= -\frac{\epsilon_2}2\omega_3(Z, \cdot)= \epsilon_1\epsilon_2 H_a d\hat{q}^a \;, \notag \\
	\theta_3^P &= \frac{\epsilon_2}2\omega_2(Z, \cdot)=- \epsilon_2 2q^a \Omega_{ab} d\hat{q}^b \;. 
\end{align} 
The metric $g_P=\frac{2}{f_1}\eta^2+g$ is given by
\begin{equation}
	g_P = H_{ab} \left( dq^a dq^b - \epsilon_2 d\hat{q}^a d\hat{q}^b \right)
	+ \epsilon_1 \frac{2}{(2H + c)} \left(\frac{1 }2 d\tilde{\phi} + \hat{q}^a \Omega_{ab} d\hat{q}^b - \epsilon_2 {q}^a \Omega_{ab} d{q}^b  \right)^2 \;.
\end{equation}
A degenerate tensor field $\tilde{g}$ on $P$ that restricts to the $\varepsilon$-quaternionic K\"ahler metric $g'$ given in Eq.\ \eqref{eqDefQKMetricHKQK} on any appropriate submanifold $M'$ is given by
\begin{align}
	\epsilon_1  2\sigma \tilde{g} &= \tilde{H}_{ab} \left( dq^a dq^b - \epsilon_2 d\hat{q}^a d\hat{q}^b  \right) \notag \\
	&\hspace{2em} + \epsilon_1 \epsilon_2  \frac{8}{(2H - c)^2} \left( q^a \Omega_{ab} d\hat{q}^b \right)^2
	-\epsilon_1 \frac{4}{2H(2H - c)} (q^a\Omega_{ab} dq^b)^2 \notag \\
	&\hspace{2em} - \epsilon_1 \frac{8H}{(2H - c)^2(2H + c)} \left[ \left(\frac{1}2 d\tilde{\phi} + \hat{q}^a \Omega_{ab} d\hat{q}^b\right) + \epsilon_2 \frac{c}{2H} {q}^a \Omega_{ab} d{q}^b \right]^2\;, 
	\label{eq:gprime}
\end{align}
where $\tilde{H}_{ab} := \frac{\partial^2}{\partial q^a \partial q^b} \tilde{H}$, and  $\tilde{H} := -\frac12 \log (2H - c)$. The horizontal lift of $Z \in \mathfrak{X}(M)$ to $\tilde{Z} \in \mathfrak{X}(P)$ is given by 
\[
	\tilde{Z} = Z - \eta(Z) X_P = - \epsilon_1 H_a \Omega^{ab} \frac{\partial}{\partial q^b} - \epsilon_1 2H X_P\;.
\]
The fundamental vector field is given by $X_P = - \epsilon_2 2 \frac{\partial}{\partial \tilde{\phi}}$, and therefore the vector field $Z_1^P \in \mathfrak{X}(P)$ is given by 
\begin{equation}
Z_1^P = \tilde{Z} + f_1 X_P = -\epsilon_1 H_a \Omega^{ab} \frac{\partial}{\partial q^b} -\epsilon_1 \epsilon_2 2c \frac{\partial}{\partial \tilde{\phi}} \;.\label{eqZPCMap}
\end{equation}
As a corollary of Theorem \ref{mainThm}, we obtain that the restriction of the tensor field given in Eq.\ \eqref{eq:gprime} to any codimension one submanifold $M'\subset \tilde{N}$ that is transversal to the vector field given in Eq.\ \eqref{eqZPCMap} defines an $\varepsilon$-quaternionic K\"ahler metric for any $c\in \mathbb{R}$ (after restriction to open subsets where $2H-c$ and $2H+c$ have constant sign).

Setting $c = 0$ in Eq.\ \eqref{eq:gprime} reproduces the formula for the target metrics of the local spatial, temporal and Euclidean $c$-maps  in \cite[Sec.\ 4.2]{Cortes:2015wca} up to an overall factor given by\footnote{%
This implies that the reduced scalar curvatures are related by $\nu = {\epsilon_1}{2} \sigma \nu^{[CDMV]}$. From Remark \ref{rem:1} we have $\nu = -\epsilon_1 4 \sigma$, which is consistent with $\nu^{[CDMV]} = -2$.} 
$\tilde{g} = \frac{\epsilon_1\sigma}{2} g'^{[CDMV]}$.
This shows that the families of $\varepsilon$-quaternionic K\"ahler manifolds defined above describe one-parameter deformations of the local spatial, temporal and Euclidean $c$-map metrics.

\subsection{Ferrara-Sabharwal form}

In this subsection we will write Eq.\ \eqref{eq:gprime} in an alternative system of coordinates.
This will, in particular, make manifest that for the case $(\epsilon_1, \epsilon_2) = (-1, -1)$ the quaternionic K\"ahler metric obtained from Eq.\ \eqref{eq:gprime} agrees with the one-loop deformation \cite{RoblesLlana:2006ez} of the original Ferrara-Sabharwal $c$-map metric \cite{FS}. 

We use the following system of coordinates on the conical affine special $\epsilon_1$-K\"ahler manifold $M$:
\begin{equation}
\left(r=\sqrt{X^IN_{IJ}\bar{X}^J},~ \phi:=\mathrm{arg}\,X^0=-\frac{\epsilon_1i_{\epsilon_1}}{2}(\mathrm{log}\,\bar{X}^0-\mathrm{log}\,X^0),~ z^\mu=\frac{X^\mu}{X^0}\right)_{\mu=1,\ldots,n}.
\end{equation}
The inverse coordinate transformation is given by $X^I=\frac{r e^{i_{\epsilon_1}\phi}}{\sqrt{z^IN_{IJ}\bar{z}^J}}z^I=re^{\mathcal{K}/2}e^{i_{\epsilon_1}\phi}z^I$, where $z^0:=1$. In these coordinates, we have $J\xi=\partial_\phi$ and \begin{equation}
\tilde{\eta}=\frac{1}{r^2}g_M(J\xi,\cdot)=-\epsilon_1d\phi-\frac{1}{2}d^c\mathcal{K},
\end{equation}
 where $d^c=i_{\epsilon_1}(\bar{\partial}-\partial)$.
 
Now, we translate each term in Eq.\ \eqref{eq:gprime} into the set of coordinates $(\rho,\tilde{\varphi},z^\mu,\tilde{\zeta}_I,\zeta^J)$ used in \cite[Eq. (4.11)]{ACDM} (generalised to the case where $z^\mu$ may be para-holomorphic). Let us define
\begin{equation}
\rho:=r^2-c=2H-c\;,
\qquad 
\tilde{\varphi}:=-2\tilde{\phi} \;,
\qquad 
(p_a)=(2\Omega_{ab}\hat{q}^b)=(\tilde{\zeta}_I,\zeta^J) \;.
\end{equation}
In these coordinates the vector field $Z_1^P$ in Eq.\ \eqref{eqZPCMap} is given by $Z_1^P=-\epsilon_1 2\frac{\partial}{\partial\phi}-\epsilon_1\epsilon_2 2c\frac{\partial}{\partial\tilde{\phi}}$. Hence, we can choose $M':=\{\phi=0\}\subset P$ as a codimension one submanifold transversal to $Z_1^P$.

Using the fact that $\zeta^I d\tilde{\zeta}_I-\tilde{\zeta}_Id\zeta^I=-4\hat{q}^a\Omega_{ab}d\hat{q}^b$ and
\begin{equation}
d^c\mathcal{K}=-2\tilde{\eta}\big|_{M'}=\frac{2\epsilon_1}{r^2}\omega_M(\xi,\cdot)\Big|_{M'}=-\frac{2\epsilon_1}{H}q^a \Omega_{ab}dq^b\Big|_{M'},\label{eqDCK}
\end{equation}
the last term in Eq.\ \eqref{eq:gprime} is given by
$-\epsilon_1\frac{1}{4\rho^2}\frac{\rho+c}{\rho+2c}(d\tilde{\varphi}+\zeta^Id\tilde{\zeta}_I-\tilde{\zeta}_Id\zeta^I+\epsilon_1\epsilon_2\,c\,d^c\mathcal{K})^2$
after restricting to $M'$. Next, we calculate (using $dr^2=\frac{1}{4(\rho+c)}d\rho^2$, $dH=\frac{1}{2}d\rho$ and Eqs.\ \eqref{eqCASKmetricDecomp}, \eqref{eqDCK})
\begin{align}\nonumber
&\qquad\tilde{H}_{ab}dq^adq^b-\frac{2\epsilon_1}{H(2H-c)}(q^a\Omega_{ab}dq^b)^2\Bigg|_{M'}\\\nonumber
&=-\frac{1}{2H-c}g_M+\frac{2(dH)^2}{(2H-c)^2}-\frac{2\epsilon_1}{H(2H-c)}(q^a\Omega_{ab}dq^b)^2\Bigg|_{M'}\\
&=\frac{\rho+2c}{4\rho^2(\rho+c)}d\rho^2+\frac{\rho+c}{\rho}g_{\bar{M}}\;.
\end{align}
To translate the remaining terms, let us define one-forms $A_I:=d\tilde{\zeta}_I+F_{IJ}d\zeta^J$ and note that $H_{ab}d\hat{q}^ad\hat{q}^b=-\epsilon_1H^{ab}dp_a dp_b=-\epsilon_1 A_IN^{IJ}\bar{A}_J$. Also note that, similarly to \cite[Lemma 3]{ACDM}, one can prove that \begin{equation}
-A_IN^{IJ}\bar{A}_J+\frac{2}{\rho+c}(X^IA_I)(\bar{X}^I\bar{A}_I)=-\frac{\epsilon_1}{2}\hat{H}^{ab}dp_a dp_b \;,
\end{equation} where
$(\hat{H}^{ab})=\begin{pmatrix}\mathcal{I}^{-1}&\mathcal{I} ^{-1}\mathcal{R}\\\mathcal{R}\mathcal{I}^{-1}& -\epsilon_1\mathcal{I}+\mathcal{R}\mathcal{I}^{-1}\mathcal{R}\end{pmatrix}$ is defined in terms of the real matrix-valued functions $\mathcal{R}=(\mathcal{R}_{IJ})$, $\mathcal{I}=(\mathcal{I}_{IJ})$ defined by
\begin{equation}
\mathcal{N}_{IJ}:=\mathcal{R}_{IJ}+i_{\epsilon_1}\mathcal{I}_{IJ} :=\bar{F}_{IJ}-\epsilon_1i_{\epsilon_1}\frac{N_{IK}X^KX^LN_{LJ}}{X^KN_{KL}X^L} \;.
\end{equation}
Note that $\mathcal{I}$ and $\mathcal{R}$ are well-defined both on $M$ and on $\bar{M}$.
Using this information we calculate 
\begin{align}\nonumber
&\qquad -\epsilon_2\tilde{H}_{ab}d\hat{q}^ad\hat{q}^b+\epsilon_1\epsilon_2\frac{8}{(2H-c)^2}(q^a\Omega_{ab}d\hat{q}^b)^2\\\nonumber
&=\frac{\epsilon_2}{2H-c}H_{ab}d\hat{q}^ad\hat{q}^b-\frac{2\epsilon_2}{(2H-c)^2}(H_ad\hat{q}^a)^2+\frac{8\epsilon_1\epsilon_2}{(2H-c)^2}(q^a\Omega_{ab}d\hat{q}^b)^2\\\nonumber
&=-\frac{\epsilon_1\epsilon_2}{\rho}A_I(X)N^{IJ}\bar{A}_J(\bar{X})-\frac{2\epsilon_2}{\rho^2}(\mathrm{Im}\, X^IA_I)^2+\frac{2\epsilon_1\epsilon_2}{\rho^2}(\mathrm{Re}\, X^IA_I)^2\\\nonumber
&=-\frac{\epsilon_1\epsilon_2}{\rho}A_I(X)N^{IJ}\bar{A}_J(\bar{X})+\frac{2\epsilon_1\epsilon_2}{\rho^2}(X^IA_I)(\bar{X}^I\bar{A}_I)\\\nonumber
&=-\frac{\epsilon_2}{2\rho}\hat{H}^{ab}dp_adp_b+\frac{2\epsilon_1\epsilon_2\,c}{\rho^2(\rho+c)}|X^IA_I(X)|^2\\
&=-\frac{\epsilon_2}{2\rho}\hat{H}^{ab}dp_adp_b+\frac{2\epsilon_1\epsilon_2\,c}{\rho^2}e^{\mathcal{K}}|z^IA_I(z)|^2 \;.
\end{align}
Putting everything together, we find that the expression given in Eq.\ \eqref{eq:gprime} restricts to the following metric on $M'$:
\begin{align}g^c_{\varepsilon FS}=\frac{\rho+c}{\rho}g_{\bar{M}}&+\frac{1}{4\rho^2}\frac{\rho+2c}{\rho+c}d\rho^2-\epsilon_1\frac{1}{4\rho^2}\frac{\rho+c}{\rho+2c}(d\tilde{\varphi}+\sum(\zeta^Id\tilde{\zeta}_I-\tilde{\zeta}_Id\zeta^I)+\epsilon_1\epsilon_2\,c\,d^c\mathcal{K})^2\nonumber \\
&-\frac{\epsilon_2}{2\rho}\sum dp_a\hat{H}^{ab}dp_b+\frac{2\epsilon_1\epsilon_2\,c}{\rho^2}e^{\mathcal{K}}\left|\sum (z^Id\tilde{\zeta}_I+F_I(z)d\zeta^I)\right|^2,\label{DefFSmetric}
\end{align}
which is defined on the two domains $\{\rho>\max\{0,-2c\}\}$ and $\{-c<\rho<\max\{0,-2c\}\}$ in $\bar{M}\times \mathbb{R}^{2n+4}$, where $(\rho,\tilde{\varphi},\tilde{\zeta}_I,\zeta^J)$ are standard coordinates on the second factor. For $(\epsilon_1, \epsilon_2) = (-1, -1)$ this agrees with \cite[Eq. (4.11)]{ACDM} and, hence, with the one-loop deformed local $c$-map metric derived in \cite{RoblesLlana:2006ez}.


\begin{thebibliography}{99}

\bibitem[AC1]{AC1}
 D.V.\ Alekseevsky, and V.\ Cort\'es, 
{\it Classification of pseudo-Riemannian symmetric spaces of quaternionic K\"ahler type},
Amer. Math. Soc. Transl. {\bf 2} 213 (2005), 33-62.

\bibitem[AC2]{AC2} D.V.\ Alekseevsky, and V.\ Cort\'es, {\it 
The twistor spaces of a para-quaternionic K\"ahler manifold},  Osaka J.\ Math.\ {\bf 45} 
(2008), no.\ 1, 215--251.


\bibitem[ACM]{ACM} D.V.\ Alekseevsky, V.\ Cort\'es and T.\ Mohaupt, {\it 
Conification of K\"ahler and hyper-K\"ahler manifolds},  Commun.\ Math.\ Phys.\ {\bf 324} 
(2013), no.\ 2, 637--655 [arXiv:1205.2964].

\bibitem[ACDM]{ACDM}  D.V.\ Alekseevsky, V.\ Cort\'es, M.\ Dyckmanns and T.\ Mohaupt, {\it
Quaternionic K\"ahler metrics associated with special K\"ahler manifolds}, J.\ Geom.\ Phys.\ {\bf 92} (2015), 271--287 [arXiv:1305.3549].

\bibitem[APP]{APP} S.\ Alexandrov,  D.\ Persson and B.\ Pioline,
{\it  Wall-crossing, Rogers dilogarithm, and the QK/HK correspondence}, 
JHEP {\bf 1112} 027 (2011) [arXiv:1110.0466].

\bibitem[B]{B} F.\ Battaglia, {\it Circle Actions and Morse Theory on
Quaternion-K\"ahler Manifolds}, J.\ London Math.\ Soc.\ {\bf 59} (1999), no.\ 1, 345--358.   

\bibitem[CDMV]{Cortes:2015wca}
  V.~Cort\'es, P.~Dempster, T.~Mohaupt and O.~Vaughan,
  \emph{Special Geometry of Euclidean Supersymmetry IV: the local c-map},
  JHEP {\bf 1510} (2015) 066
  [arXiv:1507.04620].  

\bibitem[CM]{CM}
  V.~Cort\'es, and T.~Mohaupt,
  \emph{Special Geometry of Euclidean Supersymmetry III: the local r-map, instantons and black holes},
  JHEP {\bf 0907} (2009) 066
  [arXiv:0905.2844].
  
\bibitem[CMMS]{CMMS}
  V.~Cort\'es, C.\ Mayer, T.\ Mohaupt and F.\ Saueressig,
  \emph{Special geometry of Euclidean Supersymmetry II: hypermultiplets and the c-map},
  JHEP {\bf 0506} (2005) 025
  [arXiv:hep-th/0503094].
  
	
\bibitem[DJS]{Dancer:2004}
  A.~S.~Dancer, H.~R.~J{\o}rgensen, A.~F.~Swann,
	\emph{Metric Geometries over the Split Quaternions}, Rend.\ Sent.\ Mat.\ Univ.\ Politec.\ Torino {\bf 63} (2005), no.\ 2, 119--139 
	[arXiv:math/0412215].
  
\bibitem[D]{Dyckmanns:2015} M.\ Dyckmanns, {\it The hyper-K\"ahler/quaternionic K\"ahler correspondence and the geometry of the c-map}, PhD thesis, University of Hamburg, 2015 [\href{http://ediss.sub.uni-hamburg.de/volltexte/2015/7542}{http://ediss.sub.uni-hamburg.de/volltexte/2015/7542}].

\bibitem[FS]{FS}
  S.~Ferrara and S.~Sabharwal, {\it Quaternionic Manifolds for Type II Superstring Vacua of Calabi-Yau Spaces,}
  Nucl.\ Phys.\ B {\bf 332} (1990) 317.
   
\bibitem[GL]{GL}
K.~Galicki and H.~B.~Lawson Jr.,
{\it Quaternionic reduction and Quaternionic Orbifolds},
Math.\ Ann.\ {\bf 282} (1988), no.\ 1, 1--21.

\bibitem[Ha]{Ha} A.\ Haydys, {\it Hyper-K\"ahler and quaternionic 
K\"ahler manifolds with $S^1$-symmetries}, J.\ Geom.\ 
Phys.\ {\bf  58} (2008), no.\ 3, 293--306 [arXiv:0706.4473].

\bibitem[Hi]{Hi} N.\ Hitchin, {\it On the hyperk\"ahler/quaternion K\"ahler correspondence}, 
Commun.\ Math.\ Phys.\ \textbf{324} (2013), no.\ 1, 77��--106.


\bibitem[MS1]{MS1}
 O.~Macia, and A.~Swann,
 {\it Elementary deformations and the hyperK\"ahler-quaternionic K\"ahler correspondence}, Real and Complex Submanifolds (Daejeon, Korea, August 2014), 339--347, Y.\ J.\ Suh et al.\ (eds.), Springer Proceedings in Mathematics and Statistics 106, Springer Japan, Tokyo 2014 [arXiv:1404.1169].
 
\bibitem[MS2]{MS2}
 O.~Macia, and A.~Swann,
 {\it Twist geometry of the c-map},
 Commun.\ Math.\ Phys.\  {\bf 336} (2015), no.\ 3, 1329--1357 [arXiv:1404.0785].
 
\bibitem[M]{M}
S.\ Marchiafava, {\it Submanifolds of (para-)quaternionic
K\"ahler manifolds}, Note Mat.\ {\bf 1} (2008), no.\ 1, 295--316.

\bibitem[RSV]{RoblesLlana:2006ez}
  D.~Robles-Llana, F.~Saueressig and S.~Vandoren,
  \emph{String loop corrected hypermultiplet moduli spaces},
  JHEP {\bf 0603} (2006) 081
  [arXiv:hep-th/0602164].
  
\bibitem[S]{S} A.\ Swann, {\it HyperK\"ahler and quaternionic K\"ahler geometry}, Math. Ann. {\bf 289} (1991), no.\ 1, 421--450.
	
\bibitem[V]{V}
 S.~Vukmirovi\'c,
 {\it Para-quaternionic reduction}
 [arXiv:math/0304424].


\end{thebibliography}
\end{document}